\documentclass[oneside,english]{amsart}
\usepackage[T1]{fontenc}
\usepackage[latin9]{inputenc}
\usepackage{mathrsfs}
\usepackage{amsbsy}
\usepackage{amstext}
\usepackage{amsthm}
\usepackage{amssymb}
\usepackage{wasysym}
\usepackage{esint}

\makeatletter
\numberwithin{equation}{section}
\numberwithin{figure}{section}
\theoremstyle{plain}
\newtheorem{thm}{\protect\theoremname}
\theoremstyle{definition}
\newtheorem{defn}[thm]{\protect\definitionname}
\theoremstyle{remark}
\newtheorem{rem}[thm]{\protect\remarkname}
\theoremstyle{plain}
\newtheorem{cor}[thm]{\protect\corollaryname}
\theoremstyle{plain}
\newtheorem{lem}[thm]{\protect\lemmaname}
\theoremstyle{plain}
\newtheorem{prop}[thm]{\protect\propositionname}

\usepackage{babel}

\makeatother

\usepackage{babel}
\providecommand{\corollaryname}{Corollary}
\providecommand{\definitionname}{Definition}
\providecommand{\lemmaname}{Lemma}
\providecommand{\propositionname}{Proposition}
\providecommand{\remarkname}{Remark}
\providecommand{\theoremname}{Theorem}

\begin{document}
\title[A $\sigma_{2}$ Penrose inequality for conformal AH 4-discs ]{A $\sigma_{2}$ Penrose inequality for conformal asymptotically hyperbolic
4-discs}
\author{Hao Fang }
\address{14 MacLean Hall, Department of Mathematics, University of Iowa, Iowa
City, IA, 52242}
\email{hao-fang@uiowa.edu}
\author{Wei Wei}
\address{Department of Mathematics, Nanjing University, Nanjing, P.R.China,
210093}
\email{wei\_wei@nju.edu.cn}
\thanks{H.F.'s works is partially supported by a Simons Foundation research
collaboration grant. W.W.'s works is partially supported by the Initiative
postdoctoral fund of China.}
\begin{abstract}
In this paper, we consider conformal flat metrics on $\mathbb{R}^{4}$
with an asymptotically hyperbolic (AH) end and possible isolated conic
singularities. We define a mass term of the AH end. If the $\sigma_{2}$
curvature has lower bound $\sigma_{2}\geq\frac{3}{2}$, we prove an
inequality relating the mass and contributions from singularities.
We also classify sharp cases, which is the standard hyperbolic 4-space
$\mathbb{H}^{4}$ when no singularity occurs. It is worth noting that
our curvature condition implies non-positive energy density. 
\end{abstract}

\maketitle

\section{Introduction}

In this paper, we prove a sharp mass inequality for certain asymptotically
hyperbolic(AH) 4-manifolds. Our work is motivated by research works
in mathematical general relativity and conformal geometry of $\sigma_{k}$
curvature.

Positive mass theory is one of the central problems in geometry. Studies
of asymptotically flat manifolds lead to positive mass theorems established
by Schoen and Yau \cite{SY1,SY2}, and then Witten \cite{Witten}.
When certain minimal surfaces, or black holes, are present, Riemannian
Penrose inequalities have been established by Huisken-Ilmanen \cite{HI}
and Bray \cite{Bray1,Bray2,BL}. In particular, for both positive
mass theorems and Penrose inequalities, boundary cases are well understood
and can be determined as Euclidean spaces and Schwarzschild spaces,
respectively. These fundamental geometric results are based on an
important assumption that comes from the physics consideration. Namely,
the positive energy density condition in general relativity leads
to proper local curvature constraints, which, in the Riemannian case,
indicates non-negativity of the scalar curvature.

For universe models with a negative cosmological constant, the corresponding
mathematical theory has also been considered. See, for example \cite{HT,WXZ}.
When restricted to the Riemannian case, it involves studies of asymptotically
hyperbolic manifolds of dimension $n$. Mainly, rigidity and positive
mass theorems can be properly stated and proved for AH manifolds.
See, for example, Min-Oo \cite{MO}, Anderson-Dahl \cite{AD}, Wang
\cite{WangX}, Chru$\acute{s}$ciel-Herzlich \cite{CH}. The positive
energy density condition is also required in these results, which
is equivalent to the geometric condition that the scalar curvature
has a negative lower bound.

Conformal geometry regarding the so-called $\sigma_{k}$ curvature
is another source of our motivation. As a natural extension of the
scalar curvature, $\sigma_{k}$ curvature was first studied by Viaclovsky
\cite{V1}. It has been then extensively studied as important type
of fully non-linear PDEs with significant geometric applications.
See \cite{CGY2,CGY3,CHY,CHY2,CQY,FW1,G0,G1,G2,GB,GLW,GS,GV0,GV1,GV2,GV3,GV4,GV5,GW1,GW2,H,HLT,L1,L2,LiLuc1,LiLuc2,LL1,LL2,PVW,STW,TW,V1,V2,V3,V4,W1}
for some incomplete references in this field. In particular, works
of Chang-Gursky-Yang \cite{CGY1,CGY4} explore properties of $\sigma_{2}$
curvature in closed 4-manifolds and give a conformal characterization
of 4-spheres. Majority of geometric application in this direction
requires a so-called positive $\sigma_{k}$ cone condition. In particular,
this condition implies that the scalar curvature is point-wise positive.

To introduce our results, let us first fix notations. Let $\{x^{1},\cdots,x^{4}\}$
be the standard coordinate system on $\mathbb{R}^{4}$ and $g_{E}$
be the Euclidean metric. Let $r=\sqrt{\sum(x^{i})^{2}}$ and $\text{\ensuremath{\theta\in S^{3}} }$
be the standard polar coordinate of $\mathbb{R}^{4}.$ Denote $D=\left\{ r<1\right\} $
to be the unit ball. For future use, we also define 
\[
s=\log(\frac{1}{r}),
\]
which is a defining function of $\partial D$, the boundary of $D$.
We consider a conformal metric on $D$ as $g=\exp(2u)g_{E}$. Let
$R_{g},$ Ric$_{g}$, and $A_{g}$ be the corresponding scalar curvature,
Ricci tensor and Schouten tensor, respectively. Let $\sigma_{2}(g)=\sigma_{2}(g^{-1}A_{g})$
be the second symmetric polynomial of eigenvalues of $A_{g}$ with
respect to $g$. For example, for the standard hyperbolic space $\mathbb{H}^{4}$,
we have 
\begin{align}
u(x) & =s-\log\sinh s,\nonumber \\
{\rm Ric}_{g} & =-3g,\nonumber \\
A_{g} & =-\frac{1}{2}g,\nonumber \\
R_{g} & =6\sigma_{1}(g^{-1}A_{g})=-12,\nonumber \\
\sigma_{2}(g^{-1}A_{g}) & =\frac{3}{2}.\label{eq:sigma2}
\end{align}

\begin{defn}
\label{def:AHS}Let $p_{1},\cdots,p_{k}\in D$ be $k$ distinct points,
$k\geq0$. Let $(M,g)=(D\backslash\{p_{1},\cdots,p_{k}\},\exp(2u(x))g_{E})$
be a conformal metric on $D$. It is called a conformal asymptotically
hyperbolic space of dimension 4 with cone-like singularities if the
conformal factor $u(x)\in C^{2}(D\backslash\{p_{1,},\cdots,p_{k}\})$
satisfies the following conditions: 
\end{defn}

\begin{enumerate}
\item If $k>0,$ for each $i\in\left\{ 1,\cdots,k\right\} $ there exists
$\beta_{i}>0$ such that $|x-p_{i}|^{j}|\nabla^{k}(u(x)-\beta_{i}\log|x-p_{i}|)|$
is $C^{\alpha_{i}}(B_{\delta}(p_{i}))$ for $j=0,1,2$, and some $\alpha_{i}\in(0,1)$,
$\delta>0$;
\item Near $\partial D$ we have the following asymptotic behavior of $u$,
\label{asymp-assumption} 
\[
u(x)=s-\log\sinh s+s^{4}f(\theta)+h(r,\theta),
\]
where $f\in C^{2}(S^{3})$, and for some positive constant $C$
\[
\overline{\lim}_{s\rightarrow0}\frac{|h|+s|\nabla_{r}h|}{s^{4}}=0,\,\,\,\overline{\lim}_{s\rightarrow0}|\nabla_{rr}h|s^{6}\le C,
\]
\[
\overline{\lim}_{s\rightarrow0}\frac{|\nabla_{\theta}h|+|\nabla_{\theta\theta}h|}{s^{4}}\le C.
\]
\end{enumerate}
Note that for any metric satisfying Definition \ref{def:AHS}, 
\[
g=(\frac{2}{1-|x|^{2}})^{2}[1+2(1-|x|)^{4}f(\theta)+\hat{h}(r,\theta)]g_{E},
\]
where $\overline{\lim}_{r\rightarrow1}\frac{\hat{h}(r,\theta)}{(1-|x|)^{4}}=0$.
Thus, it is asymptotically hyperbolic near the boundary $\partial D$.
Our definition is similar to those defined in the literature. See,
for example, \cite{DGS,WangX}. 

In our setup, we allow the existence of isolated conic singularities.
For each singular point $p_{i},$ the tangent cone of the manifold
is a cone of angle $(1+\beta_{i})|S^{3}|$.

Now we define a mass quantity for the hyperbolic end. 
\begin{defn}
\label{def:m2}For $(M,g)$ satisfying Definition \ref{def:AHS},
the mass for $M$ is the following quantity 
\[
m(M,g)=m(M)=\frac{1}{|S^{3}|}\int_{S^{3}}f(\theta)d\theta,
\]
where $d\theta$ is the standard volume form for the unit sphere $S^{3}$,
and $|S^{3}|$ is the corresponding volume.
\end{defn}

\begin{rem}
In Wang \cite{WangX} and Chru$\acute{s}$ciel-Herzlich \cite{CH},
the mass of a general AH manifold is defined as the Minkowski norm
of certain vector constructed via asymptotic of the metric near the
AH end. The positive mass inequality is then established assuming
$R_{g}\ge-n(n-1)$ and the manifold being spin. The mass that we have
defined is similar but in a conformally flat setting, which is more
restrictive. In fact, we choose our sign convention so that up to
a positive constant, it is one particular component of the mass vector
in \cite{WangX}. We use our notations out of mathematical convenience. 
\end{rem}

In this paper, we discuss conformally flat AH spaces of dimension
4 with a $\sigma_{2}$ curvature positive lower bound condition. Our
main result is the following 
\begin{thm}
\label{thm:main}Assume that $(M,g)$ satisfies Definition \ref{def:AHS}.
If we further assume that 
\begin{equation}
\sigma_{2}(g^{-1}A_{g})\geq\frac{3}{2},\label{eq:>}
\end{equation}
then we have the following 
\begin{equation}
-m(M,g)\geq F(\beta_{1},\cdots,\beta_{k})\geq0,\label{eq:M1}
\end{equation}
where $F(\beta_{1},\cdots,\beta_{k})=\tilde{\beta}^{2}(\tilde{\beta}+2)^{2}+(\frac{8}{3}\tilde{\beta}+4)(\sum_{i=1}^{k}\beta_{i}^{2}-\tilde{\beta}^{2})$
with $\widetilde{\beta}:=\bigg(\sum_{i=1}^{k}\beta_{i}^{3}\bigg)^{1/3}.$
In particular, when $k=1$, or $M$ has exactly one singular point,
we have 
\begin{equation}
-m(M,g)\geq\frac{1}{20}(\beta_{1}+2)^{2}\beta_{1}^{2}.\label{eq:M2}
\end{equation}
If the equality in (\ref{eq:M1}) holds, then $u$ is rotationally
symmetric and $k=1$; Furthermore, $(M,g)$ can be identified as the
Chang-Han-Yang model, which will be discussed in Section 2. 
\end{thm}

As a consequence, we have the following special case
\begin{cor}
\label{cor:main cor}Conditions are given as in Theorem \ref{thm:main}.
If $M$ is smooth without singular points, then 
\[
m(M,g)\leq0.
\]
In particular, the equality holds if and only if $(M,g)$ is the standard
hyperbolic space, $\mathbb{H}^{4}.$ 
\end{cor}

We make some comments regarding our results.

Positive mass problems related to $\sigma_{k}$ curvature have been
considered for both asymptotic flat and asymptotic hyperbolic manifolds
under different settings. See Ge-Wang-Wu \cite{GWW1,GWW2,GWW3,GWW4},
Ge-Wang-Wu-Xia \cite{GWWX1} and Li-Nguyen\cite{LiLuc1}. Our definitions
and results are different in flavor. Also, we focus only on the $\sigma_{2}$
curvature in dimension 4 case.

The most interesting feature of our results is our curvature assumption.
A simple computation shows that our assumption $\sigma_{2}(g^{-1}A_{g})\geq\frac{3}{2}$
leads to the scalar curvature condition $R_{g}\leq-12$, which is
exactly the opposite comparing to that posed in previous works of
\cite{WangX} and Chru$\acute{s}$ciel-Herzlich \cite{CH}. In a vague
sense, we are considering a class of AH manifolds with negative or
non-positive energy density. Theorem \ref{thm:main} and Corollary
\ref{cor:main cor} should be viewed as negative mass theorems under
these assumptions, which are reasonable. Furthermore, it is interesting
to interpret the contribution of isolated singularities, which are
right hand side terms of (\ref{eq:M1}) and (\ref{eq:M2}). We will,
however, leave any possible physics implication of our results to
experts.

From a geometric point of view, we study metrics in the so-called
negative cone, which means that in our settings the scalar curvature,
$R_{g},$ is strictly negative. Comparing to results in the positive
cone case, there are relatively few works for the negative cone case.
See \cite{MP,GV1,GG,GLL}. The difficulty for the negative cone is
mainly due to the lack of interior $C^{2}$ estimate, which plays
a significant role in fully nonlinear elliptic equations including
the $\sigma_{k}$ Yamabe problem. The counterexample of interior $C^{2}$
estimate has been constructed by Sheng-Trudinger-Wang \cite{STW}.
Our result can be viewed as a necessary condition in further study
of general $\sigma_{2}$ Nirenberg type problem in a similar setting.

Also, the geometry of sharp cases of our inequalities is first described
in Chang-Han-Yang \cite{CHY}. In particular, we are able to characterize
the standard hyperbolic space, $\mathbb{H}^{4}$, among smooth conformal
AH balls in dimension 4 with a positive $\sigma_{2}$ curvature condition.
Comparing Corollary \ref{cor:main cor} to the Chang-Gursky-Yang's
conformal 4-sphere theorem \cite{CGY1}, it is interesting to see
that $\sigma_{2}$ curvature in dimension 4 carries particularly strong
conformal geometric information to characterize space forms.

From an analytical point of view, our approach to prove Theorem \ref{thm:main}
is heavily relying on our previous work \cite{FW1}, where the $\sigma_{2}$
Yamabe problem is studied for conic 4-spheres. Instead of $\sigma_{2}$
curvature being a positive constant, which is discussed in \cite{FW1},
we find out that in the current negative cone setup, $\sigma_{2}$
curvature positive lower bound condition (\ref{eq:>}) can be used
to construct a quasi-local mass along level sets of the conformal
factor, which is similar to geometric flow methods considered by in
\cite{HI}. The monotonicity of the new quasi-local mass is established
using the delicate divergence structure of the $\sigma_{2}$ curvature
in dimension 4 and the iso-perimetric inequality for Euclidean spaces.
Generalization to non-conformally flat and higher dimensional cases
will be difficult but interesting.

The rest of the paper is organized as follows. In Section 2, we discuss
the Chang-Han-Yang ODE model for constant $\sigma_{2}$ curvature
metric and derive our main result in the special case where rotational
symmetry is assumed. In Section 3, we follow construction in \cite{FW1}
to define a quasi-local mass along level sets of the conformal factor
and prove its monotonicity. In Section 3, we study asymptotic behaviors
of our quasi-local mass near naked singular points as well as near
the hyperbolic end. As a consequence, we derive our main theorems.

The first named author would like to thank Xiao Zhang for valuable
discussion on topics in general relativity. Both authors would like
to thank Pedro Valentin De Jesus, Mijia Lai and Biao Ma for discussion.
Both authors would thank anonymous referees for suggestions and corrections
that have improved the accuracy and readability of our article.

\section{Chang-Han-Yang model and related analysis}

In this section, we first discuss the Chang-Han-Yang model of asymptotically
hyperbolic manifolds with constant $\sigma_{2}$ curvature. Then,
we briefly discuss a special case of our main result, when the conformal
factor $u$ is rotationally symmetric with respect to the origin.
And under this case, our non-linear problem gets greatly simplified.
This serves also as the model case of our analysis. Note that the
symmetry assumption on $u$ indicates the existence of possible one
singular point at the origin. It also includes the case when no singularity
exists.

First we have the following computation of $\sigma_{2}$ curvature
under the rotational symmetry condition $u(x)=u(r)$. Using the variable
$s=\log\frac{1}{r}$, $s>0$, and $v(s)=u(s)-s,$ we have the following
local metric 
\[
g=\exp(2v)(ds^{2}+dg_{S^{3}})
\]
with 
\[
\sigma_{2}(g^{-1}A_{g})=\frac{3}{2}(v_{s}^{2}-1)v_{ss}e^{-4v}.
\]

The Chang-Han-Yang model, first given in \cite{CHY}, is represented
by the solution of $\sigma_{2}=\frac{3}{2}$ under this setting, which
can be written as 
\begin{equation}
(v_{s}^{2}-1)v_{ss}=e^{4v}.\label{eq:CHY}
\end{equation}
It is easy to see that (\ref{eq:CHY}) has the following first integral:
\begin{equation}
(v_{s}^{2}-1)^{2}-e^{4v}=k^{2},\label{eq:first integral}
\end{equation}
for some non-negative constant $k$. We summarize properties of Chang-Han-Yang
solution in the following 
\begin{lem}
(\ref{eq:first integral}) has a solution, called the Chang-Han-Yang
AH solution, such that 
\end{lem}

\begin{enumerate}
\item for $s\in[0,\infty)$, $v_{s}<-1$ and $v_{ss}>0;$ 
\item when $s\to0^{+}$, $v(s)\to\infty,$ and $v_{s}\to-\infty$. Furthermore,
when $s\to0^{+},$ there is $k\geq0$ such that 
\[
v(s)=-\log\sinh s-\frac{k^{2}}{20}s^{4}+O(s^{5});
\]
\item when $s\to\infty,$ $v(s)\to-\infty$, and $v_{s}\to-\sqrt{k+1}$; 
\item $(D,e^{2u}g_{E})$ is asymptotically hyperbolic and has $m(M,e^{2u}g_{E})=-\frac{k^{2}}{20}.$ 
\end{enumerate}
Now we discuss a rotationally symmetric metric satisfying Definition
\ref{def:AHS}. Assume that $\sigma_{2}\geq\frac{3}{2}$, which means
\begin{equation}
(v_{s}^{2}-1)v_{ss}\geq e^{4v}.\label{eq:b1}
\end{equation}
Considering the fact that $g$ is asymptotically hyperbolic, it is
clear to see that 
\[
v_{s}^{2}>1,\ \ v_{ss}>0.
\]
Noting the asymptotic behavior of $v_{s}$ as $s\to\infty,$ 
\[
\lim_{s\to\infty}v_{s}(s)=-(\beta+1),
\]
where $\beta\geq0$. Since $v_{ss}>0$, for $s>0,$ 
\begin{equation}
v_{s}\leq-\beta-1\leq-1.\label{eq:b2}
\end{equation}
Considering (\ref{eq:b1}) and (\ref{eq:b2}), we have 
\begin{equation}
4v_{s}(v_{s}{}^{2}-1)v_{ss}\leq4v_{s}e^{4v}.\label{eq:b3}
\end{equation}
Integrating (\ref{eq:b3}) and considering the boundary condition,
we get 
\begin{equation}
(v_{s}{}^{2}-1)^{2}-e^{4v}\geq\beta^{2}(\beta+2)^{2}.\label{eq:add1}
\end{equation}
Due to the rotational symmetry of $v$, $f$ in (\ref{asymp-assumption})
is now a constant. Thus, 
\begin{equation}
v(s)=-\log\sinh s+fs^{4}+o(s^{4}).\label{eq:b4}
\end{equation}
Now consider the following 
\[
m(s)=\frac{1}{20}[(v'^{2}-1)^{2}-e^{4v}]
\]
and its asymptotic behavior as $s\to0.$ A direct computation shows
that 
\[
\lim_{s\to0^{+}}m(s)=\lim_{s\to0}\frac{1}{20}\big\{[(-\frac{\cosh s}{\sinh s}+4fs^{3})^{2}-1]^{2}-\frac{e^{4fs^{4}}}{\sinh^{4}s}\big\}=-f=-m(M,g).
\]
Thus, from (\ref{eq:add1}) we have the following 
\begin{equation}
-m(M,g)\geq\frac{\beta^{2}(\beta+2)^{2}}{20}.\label{eq:add2}
\end{equation}
It is then clear to see that when the equality in (\ref{eq:add2})
holds, $u$ has to satisfy (\ref{eq:CHY}). In particular, when $\beta=0,$
we have obtained the standard metric on $\mathbb{H}^{4}.$

\section{Quasi-local mass via level sets}

In this section we discuss general conformally flat AH spaces in dimension
4. In particular, on manifolds with positive $\sigma_{2}$ curvature
lower bound, we define a quasi-mass quantity and prove that it is
monotone. This construction has been actually discussed in our earlier
work \cite{FW1} in a different setup. In the asymptotically hyperbolic
case, the corresponding Schouten curvature falls into the so-called
negative cone, which means $R_{g}<0$ everywhere. However, basic ideas
in \cite{FW1} still apply here. 

We start by discussion about the critical set of the conformal factor.
In this section, we suppose that $u(x)\in C^{2}(D\backslash\{p_{1,},\cdots,p_{k}\})$.
Let 
\[
\mathcal{C}=\{x\in D\backslash\{p_{1},\cdots,p_{k}\};\ \nabla u(x)=0\}
\]
 be the critical set of $u$.
\begin{lem}
\label{lem:hausdorff measure 0}Notations as above. If $\sigma_{2}(g_{u}^{-1}A_{g_{u}})$
is never vanishing,  then $\mathcal{C}$ has at most Hausdorff dimension
2.
\end{lem}

\begin{proof}
For any $P\in\mathcal{C}$, we pick a local coordinate $\{y^{1},\cdots,y^{4}\}$
on an small open set $U$ such that $P\in U\subset\mathbb{R}^{4}$,
and denote $u_{i}=\frac{\partial u}{\partial y^{i}}$ and $u_{ij}=\frac{\partial^{2}u}{\partial y^{i}\partial y^{j}}$.
Since $\sigma_{2}(g^{-1}A_{g})(P)\neq0$ and $|\nabla u|(P)=0$, we
have $\sigma_{2}(\nabla^{2}u)(P)\neq0.$ Thus, there exist $i,j\in\{1,2,3,4\}$,
$i\neq j$, such that $u_{ii}u_{jj}-u_{ij}^{2}|_{P}\neq0$. Hence,
$\nabla u_{i}(P)$ and $\nabla u_{j}(P)$ are linearly independent.
By the implicit function theorem, the set $\{x,\frac{\partial u}{\partial y^{i}}=0\,\text{and}\,\frac{\partial u}{\partial y^{j}}=0\}$
is locally smooth and of dimension 2 near $P$. Shrinking $U$ if
necessary, we have $\mathcal{C\cap}U\subset\{x,\frac{\partial u}{\partial y^{i}}=0\,\text{and}\,\frac{\partial u}{\partial y^{j}}=0\}$.
We have thus concluded our proof.
\end{proof}
For the rest of this section, consider a smooth function $u(x)\in C^{2}(D\backslash\{p_{1,},\cdots,p_{k}\})$
that satisfies conditions posed in Definition \ref{def:AHS}. We also
assume that $\sigma_{2}(g_{u}^{-1}A_{g_{u}})$ is non-vanishing. We
define $t_{0}=\inf_{D\backslash\{p_{1,},\cdots,p_{k}\}}u$. Note that
when $k\geq1,$ $t_{0}=-\infty.$ For any $t>t_{0},$ define the following
\begin{align*}
S(t) & =\{x\in D\backslash\{p_{1,},\cdots,p_{k}\},\ u(x)<t\},\\
S^{*}(t) & =S(t)\cup\{p_{1,},\cdots,p_{k}\}.
\end{align*}
It is clear that, by Definition \ref{def:AHS}, both $S(t)$ and $S^{*}(t)$
are non-empty, open and bounded. We fix the following notations for
the rest of the article:
\begin{align*}
\partial S(t) & =\overline{S(t)}\backslash S(t),\\
\partial S^{*}(t) & =\overline{S^{*}(t)}\backslash S^{*}(t),\\
L(t) & =\{x\in D,\ u(x)=t\},\\
L(t)_{0} & =L(t)\backslash\mathcal{C}.
\end{align*}
 
\begin{lem}
\label{lem:boudary}Notations and assumptions as above. For any $t>t_{0},$
$\partial S(t),$$\partial S^{*}(t)$ and $L(t)$ differ by 0-measure
sets in 3 dimesional Hausdorff measure ($\mathcal{H}^{3})$ sense.
\end{lem}

\begin{proof}
It is obvious that $\partial S(t)\backslash\text{\ensuremath{\partial S^{*}(t)}}=\{p_{1,},\cdots,p_{k}\}$,
which has $\mathcal{H}^{3}$ measure 0. Since $u$ is $C^{2}$ smooth,
it is easy to see that
\begin{equation}
\partial S^{*}(t)\subset L(t).\label{eq:k1}
\end{equation}
\begin{equation}
L(t)_{0}\subset\partial S^{*}(t).\label{eq:k2}
\end{equation}
By (\ref{eq:k1}) and (\ref{eq:k2}), 
\[
(\partial S^{*}(t)\backslash L(t))\cup(L(t)\backslash\partial S^{*}(t)\subset\mathcal{C},
\]
which, by Lemma \ref{lem:hausdorff measure 0}, has measure 0 in $\mathcal{H}^{3}$
sense. We have finished the proof.
\end{proof}
\begin{rem}
\label{rem:nonzeromeasure}By Definition \ref{def:AHS}, for $t>t_{0},$
$S^{*}(t)$ is non-empty, open and bounded. By the isoperimetric inequality
\cite{osserman}, $\partial S^{*}(t)$ has non-trivial 3-dimensional
Minkovski content. By Lemma \ref{lem:hausdorff measure 0}, $\mathcal{C}$
is locally a subset of 2-dimenional surface with trivial 3-dimensional
Minkovski content. Thus, for all $t>t_{0}$, $L(t)_{0}=L(t)\backslash\mathcal{C}$
has non-trivial 3-dimensional Minkovski content, which means that
it is non-empty. Since $L(t)_{0}$ is locally a 3-dimensional hypersurface,
it has non-trivial $\mathcal{H}^{3}$ measure. 
\end{rem}

We proceed to discuss a local coordinate near a generic point $P\in L(t)_{0}$
near which $L(t)$ is smooth. We first define one particular coordinate
function
\[
y^{4}(Q)={\rm sgn}(u(Q)-t){\rm dist}(Q,L(t))
\]
for $Q$ near $P.$ Note that this is well defined since $L(t)$ is
smooth near $P$. We also define local normal coordinate functions
$y^{1},y^{2},y^{3}$ on an open set $V\subset L(t)$ near $P$ and
then extend them smoothly to an open set $U\subset D$. Thus, we have
got a local coordinate system $\{y^{i}\},$ $i=1,\cdots,4$ of $\mathbb{R}^{4}$
near $P$ such that $<\frac{\partial}{\partial y^{i}},\frac{\partial}{\partial y^{j}}>|_{P}=\delta_{ij}$.
Note that, due to the choice of $y^{4},$ we have $<\frac{\partial}{\partial y^{4}},\frac{\partial}{\partial y^{4}}>|_{V}=1$
and $<\frac{\partial}{\partial y^{i}},\frac{\partial}{\partial y^{4}}>|_{V}=0$
for $i=1,2,3$.

We use $\nabla$ to denote the Levi-Civita connection of $g_{E}$
and write $u_{i}=\nabla_{\frac{\partial}{\partial y^{i}}}u$ and $u_{ij}=\nabla_{ij}u=\nabla_{i}\nabla_{j}u$.
By the definition of $y^{4},$ $\frac{\partial}{\partial y^{4}}|_{V}=\frac{\nabla u}{|\nabla u|}$.
We note that $u_{44}$ is independent of choices of $y^{1}$, $y^{2}$
and $y^{3}$ and well defined for generic points in $L(t)$. Let $\nabla_{ab}^{L}u$
be the Hessian of $u$ with respect to the induced metric on $L(t)$.
In the following, $\alpha,\,\beta$ range from 1 to $3$. By the definition
of $L(t),$ we have $\nabla_{\alpha}^{L}u=0$.

Let $h_{\alpha\beta}$ be the second fundamental form of the level
set $L(t)$ with respect to the outward normal vector $\frac{\nabla u}{|\nabla u|}$.
We have the following Gauss-Weingarten formula

\begin{equation}
\nabla_{\alpha\beta}u=\nabla_{\alpha\beta}^{L}u+h_{\alpha\beta}u_{4}.\label{eq:gauss-weingarten1}
\end{equation}

We may now describe the Schouten tensor using our choice of local
coordinates near $P$. Recall that for the Schouten tensor $A=A_{g},$
\[
A_{ij}=-u_{ij}+u_{i}u_{j}-\frac{|\nabla u|^{2}}{2}\delta_{ij},
\]
and near $P$, we write $g_{E}^{-1}A_{g}$ locally as a symmetric
matrix

\begin{equation}
g_{E}^{-1}A_{g}=\left(\begin{array}{cccc}
 &  &  & -\nabla_{41}u\\
 & -h_{\alpha\beta}|\nabla u|-\frac{|\nabla u|^{2}}{2}\delta_{\alpha\beta} &  & -\nabla_{42}u\\
 &  &  & -\nabla_{43}u\\
-\nabla_{41}u & -\nabla_{42}u & -\nabla_{43}u & -\nabla_{44}u+\frac{|\nabla u|^{2}}{2}
\end{array}\right).\label{eq:expression of A}
\end{equation}
For simplicity, we define a local symmetric $3\times3$ matrix 
\[
\widetilde{A}(P):=(-h_{\alpha\beta}|\nabla u|-\frac{|\nabla u|^{2}}{2}\delta_{\alpha\beta}).
\]
We also define, for $L(u(P))_{0}$ at $P$, we define the unit normal
vector $\nu=\frac{\nabla u}{|\nabla u|}$ and a direct computation
shows that the corresponding mean curvature 
\begin{equation}
H={\rm div}(\frac{\nabla u}{|\nabla u|}).\label{eq:H}
\end{equation}

For future use, we establish the following key point-wise results,
which is a consequence of the asymptotic hyperbolic condition.
\begin{lem}
\label{lem:pointwise}Let $(M,g)=(D\backslash\{p_{1},\cdots,p_{k}\},\exp(2u(x))g_{E})$
be a conformal asymptotically hyperbolic space of dimension 4 with
cone-like singularities defined as in Definition \ref{def:AHS}, if
$\sigma_{2}(A_{g})>0,$ then for any $P\in D\backslash\{p_{1},\cdots,p_{k}\}\backslash\mathcal{C},$
we have 
\[
{\rm div}(|\nabla u|^{2}\nabla u)=3|\nabla u|^{2}[\nabla_{44}u+\frac{H}{3}|\nabla u|]>0,
\]
\[
\sigma_{1}(\widetilde{A})=-H|\nabla u|-\frac{3}{2}|\nabla u|^{2}<0.
\]
\end{lem}

\begin{proof}
We may use (\ref{eq:expression of A}) and (\ref{eq:H}) to establish
identities by direct computation. Furthermore, we compute
\begin{align}
0<\sigma_{2}(A) & =\sigma_{2}(\widetilde{A})+(-\nabla_{44}u+\frac{|\nabla u|^{2}}{2})\sigma_{1}(\widetilde{A})-\sum_{a=1}^{3}(\nabla_{a4}u)^{2}\nonumber \\
 & \le\sigma_{2}(\widetilde{A})+(-\nabla_{44}u+\frac{|\nabla u|^{2}}{2})\sigma_{1}(\widetilde{A}).\label{eq:key1}
\end{align}
Note also that 
\begin{equation}
\sigma_{2}(\widetilde{A})\le\frac{\sigma_{1}^{2}(\widetilde{A})}{3}.\label{eq:key2}
\end{equation}
We then combine (\ref{eq:key1}) and (\ref{eq:key2}) to conclude
\begin{align}
\sigma_{2}(A) & \le\frac{\sigma_{1}(\widetilde{A})\sigma_{1}(\widetilde{A})}{3}+(-\nabla_{44}u+\frac{|\nabla u|^{2}}{2})\sigma_{1}(\widetilde{A})\nonumber \\
 & =\sigma_{1}(\widetilde{A})(-\frac{H}{3}|\nabla u|-\nabla_{44}u).\label{eq:add7}
\end{align}
By Lemma \ref{lem:hausdorff measure 0}, $D\backslash\{p_{1},\cdots,p_{k}\}\backslash\mathcal{C}$
is open and connected, which means both factors at the right hand
side of (\ref{eq:add7}) do not change signs. Using condition (2)
of Definition \ref{def:AHS}, we may directly verify that both quantities
are negative when $P$ is near the boundary of $D$. Thus, they are
always negative, which leads to inequality parts of the lemma.
\end{proof}
Next, we define some integral quantities. From now on we use $\fint_{L(t)}$
, $\fint_{S(t)}$ to represent $\frac{1}{|S^{3}|}\varoint_{L(t)}$,
$\frac{1}{|S^{3}|}\int_{S(t)}$ , respectively, where $|S^{3}|$ is
the volume of the unit 3-sphere and $|B^{4}|$ is the volume of the
unit ball in $\mathbb{R}^{4}$. Note that $\frac{|S^{3}|}{4}=|B^{4}|$
. We will use the standard Euclidean measure $dx$, and its induced
hyper-surface measure, $dl=dl_{t}$ on $L(t)_{0}$. We also omit them
if no confusion arises. We now define following quantities for all
$t>t_{0}$,
\[
A(t)=\fint_{S(t)}e^{4u}dx,
\]
\[
B(t)=\fint_{S(t)}dx,
\]

\[
C(t)=e^{4t}B(t).
\]

We remark here that we use a slightly different sign convention here
comparing to definitions given in \cite{FW1}, but the construction
is essentially same. When $t_{0}=\inf u$ is finite, then $u\in C^{2}(D)$
has no interior singular points, and we define $A(t_{0})=B(t_{0})=C(t_{0})=0.$
Thus, $A,B,C$ are defined for all $t\in[t_{0},\infty).$
\begin{lem}
\label{lem:ACont-1}Notations as above. Functions $A(t),B(t),C(t)$
are absolutely continuous. 
\end{lem}

\begin{proof}
We follow similar arguments in \cite{FL,FW1}. For completeness, we
sketch a proof here. By the co-area formula (see Lemma 2.3 in \cite{BZ})
and Lemma \ref{lem:hausdorff measure 0}, for any $t_{2}>t_{1}\geq t_{0},$
\[
B(t_{2})-B(t_{1})=\frac{|\mathcal{C}\cap u^{-1}(t_{1},t_{2}]|}{|S^{3}|}+\int_{t_{1}}^{t_{2}}\fint_{L(\tau)_{0}}\frac{1}{|\nabla u|}d\mathcal{H}d\tau=\int_{t_{1}}^{t_{2}}\fint_{L(\tau)_{0}}\frac{1}{|\nabla u|}d\mathcal{H}d\tau,
\]
which by the fundamental theorem of Lebesgue integral means that $B(t)$
is absolutely continuous. Thus, $C(t)=e^{4t}B(t)$ is also absolutely
continuous. A similar argument shows that $A(t)$ is absolutely continuous.
\end{proof}
\begin{cor}
\label{cor:finite perimeter}Notations and assumptions as above. There
exists a dense set $\mathcal{T}\subset[t_{0},+\infty)$ such that
$[t_{0},+\infty)\backslash\mathcal{T}$ has Lebesgue measure 0 and
for any $t\in\mathscr{\mathcal{T}},$
\begin{align*}
A'(t) & =e^{4t}\ \fintop_{L(t)_{0}}\frac{1}{|\nabla u|},\\
B'(t) & =\fintop_{L(t)_{0}}\frac{1}{|\nabla u|},\\
C'(t) & =4C+A'.
\end{align*}
In particular, $0<|L(t)_{0}|$ is finite for all $t\in\mathcal{T}.$
\end{cor}

\begin{proof}
The existence of $\mathcal{T}$and identities in the claim can be
derived from the co-area formula and direct computations. See also
\cite{FL,FW1}. In particular, for $t\in\mathcal{T},$ $B'(t)$ is
finite. Noting that $\frac{1}{|\nabla u|}$ is uniformly bounded from
below on the closed set $L(t)\supset L(t)_{0}$, we conclude that
$|L(t)|=|L(t)_{0}|=\int_{L(t)_{0}}1$ is finite for $t\in\mathcal{T}$.
\end{proof}
Now we define a few more quantities. For any $t\in\mathcal{T},$ by
Lemma \ref{lem:boudary}, Remark \ref{rem:nonzeromeasure} and Corollary
\ref{cor:finite perimeter}, $L(t)_{0}=L(t)\backslash\mathcal{C}$
has finite, non-trivial $\mathcal{H}^{3}$ measure, and we define:
\begin{align*}
z(t) & =(\fint_{L(t)_{0}}|\nabla u|^{3}dl)^{\frac{1}{3}},\\
D(t) & =\frac{1}{4}\fint_{L(t)_{0}}[2H|\nabla u|^{2}+2|\nabla u|^{3}]dl.
\end{align*}

\begin{lem}
\label{lem: derivative} Notations as above. Suppose $\sigma_{2}(g_{u}^{-1}A_{u})$
never vanishes. For all $t_{1},t_{2}\in\mathcal{T}$ such that $t_{1}>t_{2}$,
we have
\begin{align}
D(t_{1}) & =D(t_{2})+\fintop_{S(t_{1})\backslash S(t_{2})}\sigma_{2}(g^{-1}A_{g})e^{4u}dx,\label{eq:DD}\\
z^{3}(t_{1}) & =z^{3}(t_{2})+\fintop_{S(t_{1})\backslash S(t_{2})}{\rm div}(|\nabla u|^{2}\nabla u)\ dx.\label{eq:zz}
\end{align}
\end{lem}

\begin{proof}
We follow a similar argument in \cite{FW1}. For any $t\in\mathcal{T},$
by Remark \ref{rem:nonzeromeasure} and Corollary \ref{cor:finite perimeter},
$0<\mathcal{H}^{3}(L(t)_{0})<\infty.$ By the generalized Gauss-Green
theorem (see Theorem 1 in Chapter 5.8 of \cite{EG}), Lemma \ref{lem:hausdorff measure 0},
Lemma \ref{lem:boudary} and the divergence structure of $\sigma_{2}$,
we know that for $t_{1}>t_{2},$ $t_{1,}t_{2}\in\mathcal{T},$

\begin{align*}
 & \int_{S(t_{1})\backslash S(t_{2})}\sigma_{2}\left(g_{u}^{-1}A_{g_{u}}\right)e^{4u}dx\\
= & \frac{1}{4}\int_{L(t_{1})_{0}}\left[2H|\nabla u|^{2}+2|\nabla u|^{3}\right]-\frac{1}{4}\int_{L(t_{2})_{0}}\left[2H|\nabla u|^{2}+2|\nabla u|^{3}\right]\\
= & |S^{3}|(D(t_{1})-D(t_{2})),
\end{align*}
where for each $i=1,2,$ on $L(t_{i})_{0}$, $\nu=\frac{\nabla u}{|\nabla u|}$
and $H|\nabla u|^{2}={\rm div}(\frac{\nabla u}{|\nabla u|})|\nabla u|^{2}$
is well defined. We have proved (\ref{eq:DD}). 

Similarly, by the generalized Gauss-Green theorem, for $t_{1}>t_{2},$
$t_{1,}t_{2}\in\mathcal{T},$we compute that
\begin{align*}
\int_{S(t_{1})\backslash S(t_{2})}{\rm div}(|\nabla u|^{2}\nabla u) & =\int_{L(t_{1})_{0}}|\nabla u|^{2}<\nabla u,\nu>-\int_{L(t_{2})_{0}}|\nabla u|^{2}<\nabla u,\nu>\\
 & =\int_{L(t_{1})_{0}}|\nabla u|^{3}-\int_{L(t_{2})_{0}}|\nabla u|^{3}\\
 & =|S^{3}|(z(t_{1})^{3}-z(t_{2})^{3}),
\end{align*}
where for each $i=1,2,$ on $L(t_{i})_{0}$, $\nu=\frac{\nabla u}{|\nabla u|}$.
We have proved (\ref{eq:zz}). 
\end{proof}

We may now extend definitions of functions $z$ and $D$ to all $t\in[t_{0},+\infty)$.
Again, if $t_{0}>-\infty,$ we define $z(t_{0})=D(t_{0})=0.$ Consider
any $t>t_{0}$ and $t\not\in\mathcal{T}.$ Since $\mathcal{T}$ is
dense in $[t_{0},\infty),$ there exists $t'\in\mathcal{T}$ such
that $t'<t.$ We define
\begin{align*}
D(t) & =D(t')+\fintop_{S(t)\backslash S(t')}\sigma_{2}(g^{-1}A_{g})e^{4u}dx;\\
z(t) & =[z(t')^{3}+\fintop_{S(t)\backslash S(t')}{\rm div}(|\nabla u|^{2}\nabla u)\ dx]^{\frac{1}{3}}.
\end{align*}

\begin{rem}
\label{rem:full definition}Due to Lemma \ref{lem: derivative}, the
above definitions are independent of choice of $t'\in\mathcal{T}$.
In particular, (\ref{eq:DD}) and (\ref{eq:zz}) now hold for all
$t_{1}>t_{2}$, $t_{1},t_{2}\in[t_{0},+\infty).$
\end{rem}

\begin{lem}
\label{lem:ACont}Notations as above. Functions $D(t)$ and $z(t)$
are absolutely continuous. Furthermore, for a.e. $t\in[t_{0},+\infty)$
\begin{align*}
D'(t) & =\fintop_{L(t)_{0}}\frac{\sigma_{2}(g^{-1}A_{g})e^{4u}}{|\nabla u|},\\
z'(t) & =\frac{1}{3z^{2}}\fintop_{L(t)_{0}}(H|\nabla u|+3\nabla_{44}u)|\nabla u|,
\end{align*}
where $H$ is the mean curvature of $L(t)_{0}.$
\end{lem}

\begin{proof}
We apply (\ref{eq:zz}), the co-area formula, Lemma \ref{lem:hausdorff measure 0}
and Lemma \ref{lem:pointwise} to get, for any $t_{1}>t_{2}\geq t_{0}$
\begin{align}
 & z^{3}(t_{1})-z^{3}(t_{2})\nonumber \\
 & =\fint_{S(t_{1})\backslash S(t_{2})}{\rm div}(|\nabla u|^{2}\nabla u)\nonumber \\
 & =\fint_{\{t_{2}<u\le t_{1}\}\cap\mathcal{C}}{\rm div}(|\nabla u|^{2}\nabla u)+\int_{t_{2}}^{t_{1}}\fint_{(L_{\tau})_{0}}\frac{\mathrm{div}(|\nabla u|^{2}\nabla u)}{|\nabla u|}d\mathcal{H}d\tau\nonumber \\
 & =\int_{t_{2}}^{t_{1}}\fint_{(L_{\tau})_{0}}\frac{\mathrm{div}(|\nabla u|^{2}\nabla u)}{|\nabla u|}d\mathcal{H}d\tau=\int_{t_{2}}^{t_{1}}\fint_{(L_{\tau})_{0}}(H|\nabla u|+3\nabla_{44}u)|\nabla u|d\mathcal{H}d\tau\geq0.\label{eq:z increasing}
\end{align}
Thus by the fundamental theorem of Lebesgue integral calculus, $z^{3}(t)$
is absolutely continuous and non-decreasing. To prove that $z(t)$
is absolutely continuous, we only need to show that 
\begin{equation}
z(t)>0\label{eq:z>0}
\end{equation}
for all $t>t_{0}.$ By Remark \ref{rem:nonzeromeasure}, $z(t)>0$
for $t\in\mathcal{T}$, which, combined with the monotonic property
of $z$ (\ref{eq:z increasing}), leads to (\ref{eq:z>0}). 

Using (\ref{eq:DD}), a similar argument can be made to establish
the absolute continuity of $D(t)$ and compute its derivative.
\end{proof}
\begin{lem}
\label{lem:inequality}Notations as above. If we further assume that
$\sigma_{2}(g^{-1}A_{g})\geq\frac{3}{2},$ then for a.e. $t\in[t_{0},+\infty)$,
$D'(t)\geq\frac{3}{2}A'(t).$
\end{lem}

\begin{proof}
This is a direct consequence of Corollary \ref{cor:finite perimeter}
and Lemma \ref{lem:ACont}.
\end{proof}
Finally we are ready to define our quasi-local mass.
\begin{defn}
\label{def:quasi-local mass}For any $t\in[t_{0},\infty),$ we define 
\end{defn}

\[
m(t)=\frac{1}{5}[\frac{2}{3}D(t)+\frac{4}{9}D(t)z(t)+\frac{1}{36}z^{4}(t)-C(t)].
\]

\begin{prop}
\label{prop:nice m}Let $(M,g)$ be a conformally flat asymptotically
hyperbolic 4-manifold with singularities. Suppose $\sigma_{2}(g_{u}^{-1}A_{u})$
never vanishes in $D$. Then $m(t)$ is absolutely continuous. 
\end{prop}

\begin{proof}
This follows from Definition \ref{def:quasi-local mass}, Lemma \ref{lem:ACont-1}
and Lemma \ref{lem:ACont}.
\end{proof}
We proceed to establish the following crucial monotonic result, which
is a slight modification of Theorem 15 in \cite{FW1}. 
\begin{thm}
\label{monotone}Let $(M,g)$ be a conformally flat asymptotically
hyperbolic 4-manifold with possible singularities. We use notations
given as above. Suppose $\sigma_{2}(g_{u}^{-1}A_{u})\ge\frac{3}{2}$.
Then $m(t)$ is non-decreasing with respect to $t.$ That is, we have
a.e. $t\in[t_{0},\infty)$, 
\[
m'(t)\geq0.
\]
\end{thm}

\begin{proof}
This proof is similar to that of Theorem 15 in \cite{FW1}, with some
key changes. For completeness, we provide an argument here. By Proposition
\ref{prop:nice m}, we only need to compute $m'(t)$ for generic $t$
such that (\ref{eq:DD}), (\ref{eq:zz}) and formulae in Corollary
\ref{cor:finite perimeter} and Lemma \ref{lem:ACont} hold. 

From local estimates (\ref{eq:add7}) and the fact that $\sigma_{2}(A)\ge\frac{3}{2}e^{4u}$,
we derive the following using Cauchy inequality:

\begin{align}
 & \fint_{L(t)}|\sigma_{1}(\widetilde{A})|\cdot|\nabla u|dl\fint_{L(t)}|\frac{H}{3}|\nabla u|+\nabla_{44}u|\cdot|\nabla u|\nonumber \\
\ge & \bigg(\fint_{L(t)}\sqrt{-\sigma_{1}(\widetilde{A})|\nabla u|(\frac{H}{3}|\nabla u|+\nabla_{44}u)|\nabla u|}\bigg)^{2}\nonumber \\
\ge & \bigg(\fint_{L(t)}\sqrt{\frac{3}{2}e^{4u}|\nabla u|^{2}}\bigg)^{2}\nonumber \\
\ge & \frac{3}{2}e^{4t}\big(\fint_{L(t)}|\nabla u|\big)^{2}.\label{eq:key 3}
\end{align}

Noting that Corollary \ref{cor:finite perimeter}, (\ref{eq:key 3})
leads to 
\begin{align}
 & (A')^{2}\fint_{L(t)}|\sigma_{1}(\widetilde{A})||\nabla u|\cdot[\frac{1}{3}\frac{d}{dt}(z^{3})]\label{eq:key110}\\
\ge & \frac{3}{2}\big(\fint_{L(t)}|\nabla u|\big)^{2}e^{4t}\cdot e^{8t}(\fint_{L(t)}\frac{1}{|\nabla u|})^{2}\nonumber \\
= & \frac{3}{2}e^{12t}(\fint_{L(t)}|\nabla u|)^{2}(\fint_{L(t)}\frac{1}{|\nabla u|})^{2}\nonumber \\
\ge & \frac{3}{2}e^{12t}|L(t)|^{4}\frac{1}{|S^{3}|^{4}}\nonumber \\
\ge & \frac{3}{2}e^{12t}B(t)^{3}4^{4}|B^{4}|\frac{1}{|S^{3}|^{4}}\nonumber \\
= & \frac{3}{2}(4C(t))^{3},\nonumber 
\end{align}
where the third inequality is due to Cauchy inequality and the fourth
inequality holds because of the iso-perimetric inequality. By the
inequality of arithmetic and geometric means, we then derive from
(\ref{eq:key110})

\begin{align}
4C & \le\frac{1}{3}(2zA'+\frac{2}{3}z'\fint_{L(t)}|\sigma_{1}(\widetilde{A})||\nabla u|).\label{eq:11}
\end{align}

Using Corollary \ref{cor:finite perimeter} and Lemma \ref{lem:inequality},
(\ref{eq:11}) then implies that

\begin{align*}
C' & \le A'+\frac{1}{3}[2zA'-\frac{2}{3}z'\fint_{L(t)}\sigma_{1}(\widetilde{A})|\nabla u|]\\
 & \le\frac{2}{3}D'+\frac{4}{9}zD'-\frac{2}{9}z'\fint_{L(t)}\sigma_{1}(\widetilde{A})|\nabla u|\\
 & =\frac{2}{3}D'+\frac{4}{9}(zD)'-\frac{4}{9}z'D-\frac{2}{9}z'\fint_{L(t)}\sigma_{1}(\widetilde{A})|\nabla u|\\
 & =\frac{2}{3}D'+\frac{4}{9}(zD)'-\frac{4}{9}z'(-\frac{1}{4}z^{3})\\
 & =\frac{2}{3}D'+\frac{4}{9}(zD)'+\frac{1}{36}(z^{4})^{'},
\end{align*}
which is equivalent to $m'(t)\geq0.$ Noting that by Lemma \ref{lem:ACont},
$z(t)>0$, we may then use Lemma \ref{lem:inequality} effectively
to get the second inequality above. 
\end{proof}
\begin{rem}
Our proof of Theorem \ref{monotone} and that of Theorem 10 in \cite{FW1}
are very similar with two key different points. First, we assume only
the $\sigma_{2}\geq\frac{3}{2}$ here while in \cite{FW1}, $\sigma_{2}$
curvature is fixed as $\frac{3}{2}$. Second, we are working in the
negative cone case here while in \cite{FW1}, the positive cone condition
is assumed. We are thus working on different type of asymptotic profiles.
The negative cone condition here plays a crucial role to deal with
the partial differential inequality condition. 
\end{rem}

\section{Proof of Main Result}

In this section, we prove our main result. Due to Theorem \ref{monotone},
it is clear that we just need to estimate limits of our quasi-local
mass $m(t)$ as $t$ approaches extreme values. When singularity exists,
due to Definition \ref{def:AHS}, when $t$ is very small, the level
set $L(t)$ is near the singular set $\left\{ p_{i},\ i=1,\cdots,k\right\} $.
Correspondingly, when $t$ is very large, the level set $L(t)$ is
close to the boundary of disc $D$. Again, without loss of generality,
we may work only on the generic $t$. 

First, we define the following: 
\begin{defn}
Let $\beta=(\beta_{1},\cdots,\beta_{k}).$ Define 
\[
\widetilde{\beta}:=\bigg(\sum_{i=1}^{k}\beta_{i}^{3}\bigg)^{1/3}
\]
and 
\[
F=F(\beta_{1},\cdots,\beta_{k}):=\frac{1}{20}[\tilde{\beta}^{2}(\tilde{\beta}+2)^{2}+(\frac{8}{3}\tilde{\beta}+4)(\sum_{i=1}^{k}\beta_{i}^{2}-\tilde{\beta}^{2})].
\]

It particular, if $k=1,$ we have 
\[
F=\frac{1}{20}\beta_{1}^{2}(\beta_{1}+2)^{2}.
\]
When $k=0$, we define $F=0.$

We now state the following 
\end{defn}

\begin{thm}
\label{prop:positive side}If $k\geq1,$ and $u(x)\in C^{2}(D\backslash\{p_{1,},\cdots,p_{k}\})$
satisfies Condition (1) in Definition \ref{def:AHS}, we have the
following 
\begin{equation}
\lim_{t\to-\infty}m(t)=\frac{1}{20}F(\beta)\geq0;\label{eq:add4}
\end{equation}
If $u(x)\in C^{2}(D),$ and $t_{0}=\inf_{D}u$, we have 
\[
\lim_{t\to t_{0}}m(t)=0.
\]
\end{thm}

The proof of Theorem \ref{prop:positive side} follows closely a similar
argument in \cite{FW1} with some subtle changes.With the asymptotic
of $u$ given near singular points, using the divergence structure
of $\sigma_{2}$ curvature, we may prove that $\sigma_{2}(A_{g})$
is locally integrable. Then, the argument in \cite{FW1} may be used
to prove Theorem \ref{prop:positive side}. Here, we present an alternative
proof which is more direct without using divergence properties of
$\sigma_{2}(A_{g})$.

From now on, we use $C$ to denote a universal constant that depends
only on $f$, $h$ and other universal constants. We write 
\[
J=K+O(|1-r_{1}|^{k})
\]
to mean that for quantities $J$ and $K$, there exists a constant
$C$ such that $|J-K|\leq C|1-r_{1}|^{k}$. We write $K=o(1)$ to
mean $\lim_{s\to0}|K|=0.$

First, near each singular point we have the following 
\begin{lem}
\label{lem:asy of u} Denote $p_{l}=(p_{l}^{1},p_{l}^{2},p_{l}^{3},p_{l}^{4})$
for $1\le l\le k$. Assume that $u(x)\in C^{2}(D\backslash\{p_{1,},\cdots,p_{k}\})$
satisfies Condition (1) in Definition \ref{def:AHS} and we use notations
given as above. We have the following derivative estimates: for $i,j\in\{1,2,3,4\},$
as $|x-p_{l}|\rightarrow0$,

\begin{equation}
u_{i}(x)=\frac{\beta_{l}}{|x-p_{l}|^{2}}(x^{i}-p_{l}^{i})+o(\frac{1}{|x-p_{l}|}),\label{eq:firstderivative-1}
\end{equation}

\begin{equation}
u_{ij}(x)=\beta_{l}\frac{\delta_{ij}}{|x-p_{l}|^{2}}-2\beta_{l}\frac{(x^{i}-p_{l}^{i})(x^{j}-p_{l}^{j})}{|x-p_{l}|^{4}}+o(\frac{1}{|x-p_{l}|^{2}}),\label{eq:secondderivative-1}
\end{equation}

\begin{align}
H(x) & =\frac{3}{|x-p_{l}|}+o(\frac{1}{|x-p_{l}|}),\label{eq:mean curvature}
\end{align}
where $H(x)$ is the mean curvature of level set $\{x,\,u(x)=t\}$
near $p_{l}$ and $t$ is sufficiently negative. 
\end{lem}

The proof of Lemma follows a similar argument in Lemma 7 of \cite{FW1}.
Note that, local asymptotic properties given in Definition 1 is sufficient
to carry through computations and establish identies above. We omit
details of the proof here. 

To present our next lemma, we further fix some notations. Assume that
$k\geq1.$ Let $\Omega_{l}$ be connected small domain in $\mathbb{R}^{4}$
such that: $p_{l}\in\Omega_{l}$ and $\Omega_{i}\cap\Omega_{j}=\emptyset$
for any $i\neq j$. Define, for $t$ sufficiently negative, $L_{l}(t)=L(t)\cap\Omega_{l}$
which is closed. We localize the geometry near each singular point.

Fix $l\in\{1,\cdots,k\}$. We use a local polar coordinate system
near $p_{l}.$ That means, any $x\in\Omega_{l}$ can be written as
$x=r_{l}\theta,$ where $r_{l}=|x-p_{l}|$ for $l=1,\cdots,k$ and
$\theta\in S^{3}.$ Let $\pi_{l}:\mathbb{R}^{4}\backslash\{p_{l}\}\to S^{3}:x\to\pi_{l}(x)=\frac{x-p_{l}}{r_{l}}$.
Then we have the Euclidean volume form written as $dx=r_{l}^{3}dr_{l}\wedge\pi_{l}^{*}(d\theta)$.
For $t$ sufficiently negative, let $i_{l,t}$ be the inclusion map
$i_{l,t}:L_{l}(t)\to\Omega_{l}\subset\mathbb{R}^{4}.$ Let $dl_{l,t}$
be the volume form of $L_{l}(t)$ and $n_{l,t}$ be the outward normal
vector of $L_{l}(t)\subset\mathbb{R}^{4}$. We then have 
\begin{equation}
dl_{l}=i_{l,t}^{*}(\iota(n_{l,t})dx),\label{eq:add20}
\end{equation}
where $\iota$ is the contraction map.

We may now present the following 
\begin{lem}
\label{lem:conic volume } Assume that $u(x)\in C^{2}(D\backslash\{p_{1,},\cdots,p_{k}\})$
satisfies Condition (1) in Definition \ref{def:AHS} and we use notations
given as above. Let $n'_{l,t}=\frac{x-p_{l}}{|x-p_{l}|}.$ We have

\[
dl_{l}=i_{l,t}^{*}(\iota(n_{l,t}')(dx))(1+o(1))=i_{l,t}^{*}\pi_{l}^{*}(r_{l}^{3}d\theta)(1+o(1)).
\]
\end{lem}

\begin{proof}
A direct consequence of (\ref{eq:firstderivative-1}) is that $|n_{l,t}-n'_{l,t}|=o(1)$
as $t\to-\infty.$ Noting also that $dx=r_{l}^{3}dr_{l}\wedge\pi_{l}^{*}(d\theta)$,
we get our volume form estimate by Condition (1) in Definition (\ref{def:AHS})
and standard polar coordinate computation . 
\end{proof}
We are now ready to prove Theorem \ref{prop:positive side}. 
\begin{proof}
By Lemma \ref{lem:asy of u} and Lemma \ref{lem:conic volume },

\begin{align}
\lim_{t\to-\infty}\fintop_{L_{l}(t)}|\nabla u|^{3}dl_{l} & =\lim_{t\to-\infty}\fintop_{L_{l}(t)}\frac{1}{r_{l}^{3}}\cdot(|\nabla u|r_{l})^{3}dl_{l}\nonumber \\
 & =\beta_{l}^{3}\lim_{t\to-\infty}\fintop_{L_{l}(t)}\frac{1}{r_{l}^{3}}i_{l,t}^{*}\pi_{l}^{*}(r_{l}^{3}d\theta)\nonumber \\
 & =\beta_{l}^{3}\fintop_{S^{3}}d\theta=\beta_{l}^{3}.\label{eq:add9}
\end{align}
Similarly, using (\ref{eq:mean curvature}), we get 
\begin{equation}
\lim_{t\to-\infty}\fintop_{L_{l}(t)}H|\nabla u|^{2}dl_{l}=3\beta_{l}^{2}.\label{eq:add10}
\end{equation}
Theorem \ref{prop:positive side} then follows directly from Definition
\ref{def:quasi-local mass}, (\ref{eq:add9}) and (\ref{eq:add10})
when $k\geq1.$ For $k=0,$ the proof is similar and simpler since
$u$ is smooth. We have thus finished the proof. 
\end{proof}
Second, we discuss the limit of quasi-mass as $t\to+\infty$. It is
clear that this corresponds to the limit when $r=|x|\to1^{-}.$ Our
result is summarized in the following 
\begin{thm}
\label{thm:negative side}Let $(M,g)$ be a conformally flat asymptotically
hyperbolic 4-manifold with possible singularities. With notations
given as in Section 1 and Section 2, we have 
\[
\lim_{t\to\infty}m(t)=-m(M,g)=-\frac{1}{|S^{3}|}\int_{S^{3}}f(\theta)d\theta.
\]
\end{thm}

It is clear that our main results, Theorem \ref{thm:main} and Corollary
\ref{cor:main cor} are then consequences of Theorems \ref{monotone},
\ref{prop:positive side} and \ref{thm:negative side}. In other words,
the quasi-local mass connects the information of the mass of the manifold,
$m(M)$, and local geometric information of singular points. In particular,
when no singularity exists, this gives the non-positive estimate of
$m(M)$. For the sharp case, we may examine the proof of Theorem \ref{monotone},
where all inequalities become equalities. In particular, the iso-perimetric
inequality has to be sharp. This leads to obvious geometric and analytical
consequences that all functions involved have to be rotationally symmetric
and $\sigma_{2}(g^{-1}A_{g})=\frac{3}{2}.$ Thus, we have obtained
the Chang-Han-Yang model case. In the sharp case when no singularity
exists, we have obtained the standard hyperbolic space $\mathbb{H}^{4}$.

The rest of the section is now devoted to the proof of Theorem \ref{thm:negative side}.
By definition \ref{def:AHS}, it is clear that as $t\to+\infty,$
level set $L(t)$ is convergent to $\partial D$ in the Gromov-Hausdorff
sense. We will analyze limits of geometric quantities during this
procedure in detail.

For $x=(x^{1},\cdots,x^{4})\in D$, we also use the corresponding
polar coordinate $x=r\theta$ where $r=|x|$ and $\theta\in S^{3}.$
Fix a $t\in\mathbb{R}$ such that $L(t)$ is smooth and pick $x\in L(t)$.
We define 
\[
w(x)=w(r)=\log\frac{2}{1-r^{2}}=s-\log\sinh s.
\]
By Definition \ref{def:AHS}, 
\begin{equation}
u(x)=u(r,\theta)=w(r)+f(\theta)s^{4}+h(x),\label{eq:new1}
\end{equation}
where $h(x)=o(s^{4}).$ We denote $F(r,t,\theta)=:t-(w(r)+f(\theta)s^{4}+h(x))$.
Then $F(r,t,\theta)=0$ on $L(t).$ By Condition (2) in Definition
1, we see that
\begin{align*}
F_{r} & =-w_{r}-4f(\theta)s^{3}s_{r}-h_{r}\\
 & =-\frac{2r}{1-r^{2}}-\frac{4f(\theta)(\ln r)^{3}}{r}-h_{r}\neq0
\end{align*}
near $r=1$. By the implicit function theorem, we may present $r$
as a local $C^{2}$ function of $t$ and $\theta$ near $r=1.$ We
may then write $r=r(t,\theta)$. We also define a rotationally symmetric
comparison function 
\begin{equation}
u_{1}(x)=u_{1}(r)=w(r)+(\ln r)^{4}\fintop_{S^{3}}f(\theta)d\theta.\label{eq:new2}
\end{equation}
It is clear that $u_{1}$ satisfies similar asymptotic behavior as
that of $u$. Thus, we may, at least when $s$ is small enough, define
$r_{1}=r_{1}(t)$ to be the unique value such that $u_{1}(r_{1}(t))=t$.
We also define 
\[
s'=|\ln r_{1}|.
\]
It is clear that $s'$ is dependent only on $t$ and independent of
choice of $\theta$. The following limits are clear 
\[
\lim_{t\to\infty}r_{1}(t)=1,\,\lim_{t\to\infty}r(t,\theta)=1.
\]
Furthermore, $s$ and $s'$ are bounded by $|1-r|$ and $|1-r_{1}|$,
respectively.

We first establish the following basic estimates: 
\begin{lem}
\label{lem:basic estimate} Assume that $u(x)\in C^{2}(D\backslash\{p_{1,},\cdots,p_{k}\})$
satisfies Condition (2) in Definition 1 and we use notations given
as above. For $t>>1$, we have, 
\begin{align}
1-r_{1} & =O(1-r)=O(s'),\nonumber \\
1-r & =O(s'),\ \ln(r)=O(s'),\nonumber \\
r-r_{1} & =O(s'^{5}).\label{eq:s^5}
\end{align}
\end{lem}

\begin{proof}
With a fixed $t$ that is large enough, we have 
\[
u_{1}(r_{1})=t=u(r,\theta),
\]
which, according to (\ref{eq:new1}) and (\ref{eq:new2}), implies
that 
\begin{equation}
-\ln(1-r_{1}^{2})+\ln2+\fint_{S^{3}}f(\theta)d\theta(\ln r_{1})^{4}=-\ln(1-r^{2})+\ln2+f(\theta)(\ln r)^{4}+h(x).\label{eq:key identity}
\end{equation}

We first claim that $|r-r_{1}|\le C|1-r_{1}|$ for some positive constant
$C>0$. In fact, \\

\begin{align*}
-\ln(1-r_{1}^{2})+\ln(1-r^{2}) & =f(\theta)(\ln r)^{4}-\fint_{S^{3}}f(\theta)d\theta(\ln r_{1})^{4}+h(x)=o(1).\\
\end{align*}
Noting that 
\[
\ln(1-r^{2})-\ln(1-r_{1}^{2})=\ln(1+\frac{r_{1}^{2}-r^{2}}{1-r_{1}^{2}}),
\]
we get $\frac{r_{1}^{2}-r^{2}}{1-r_{1}^{2}}=o(1)$, which implies
that 
\[
|r_{1}-r|=o(1)(1-r_{1}^{2})\le o(1)|1-r_{1}|.
\]
Similarly, we obtain $|r_{1}-r|\le o(1)|1-r|$. Thus, $1-r=1-r_{1}+r_{1}-r\le C(1-r_{1}).$
Also $1-r_{1}\le C(1-r)$.

Observing that $|\ln r_{1}|=s'$ and $s=|\ln r|=O(1-r)=O(1-r_{1})=O(s')$,
we get 
\begin{align}
 & f(\theta)(\ln r)^{4}-\fint_{S^{3}}f(\theta)d\theta(\ln r_{1})^{4}+h(x)\nonumber \\
 & =\ln(1-r^{2})-\ln(1-r_{1}^{2})\\
 & =\int_{r_{1}}^{r}\frac{d\ln(1-\kappa^{2})}{d\kappa}d\kappa=\frac{-2\kappa_{0}}{1-\kappa_{0}^{2}}(r-r_{1})\label{eq:new3}
\end{align}
with $\kappa_{0}$ between $r$ and $r_{1}$, which means that $1-\kappa_{0}^{2}=O(s)=O(s'$).
We get from (\ref{eq:new3}) that 
\begin{equation}
|r-r_{1}|=O(s'^{5}).\label{eq:b5}
\end{equation}
\end{proof}
It is clear that due to (\ref{eq:b5}), we have $O(s^{k})=O(s'^{k}).$
In the following, we do not distinguish them.

For future use, we define 
\[
\varepsilon(t,\theta):=r(t,\theta)-r_{1}(t),
\]
and obtain a sharper estimate of $\varepsilon$. To simplify the notation,
we define 
\[
\bar{f}(\theta)=f(\theta)-\fintop_{S^{3}}f,
\]
which leads to 
\[
\fintop_{S^{3}}\bar{f}=0.
\]
We now present the following consequence of Lemma \ref{lem:basic estimate}. 
\begin{cor}
\label{cor:r-r1}Assume that $u(x)\in C^{2}(D\backslash\{p_{1,},\cdots,p_{k}\})$
satisfies Condition (2) in Definition 1 and we use notations given
as above. We have 
\begin{equation}
\varepsilon(t,\theta)=-\frac{(1-r_{1}^{2})}{2r_{1}}\bar{f}(\theta)s^{4}+o(s^{5}).\label{eq:=00003D00003D00005Cedefinition}
\end{equation}
\end{cor}

\begin{proof}
By Lemma \ref{lem:basic estimate},

\begin{align*}
 & f(\theta)(\ln r)^{4}-\fint_{S^{3}}f(\theta)d\theta(\ln r_{1})^{4}\\
= & (\ln r)^{4}\big(f(\theta)-\fint_{S^{3}}f(\theta)d\theta\big)+\fint_{S^{3}}f(\theta)d\theta\ \big((\ln r)^{4}-(\ln r_{1})^{4}\big)\\
= & \bar{f}(\theta)(\ln r)^{4}+\fint_{S^{3}}f(\theta)d\theta(\ln r-\ln r_{1})\big((\ln r_{1})^{3}+(\ln r_{1})^{2}\ln r+\ln r_{1}(\ln r)^{2}+(\ln r)^{3})\\
= & \bar{f}(\theta)(\ln r)^{4}+\fint_{S^{3}}f(\theta)d\theta(\frac{r-r_{1}}{r}+O(s^{10}))O(s^{3})\\
= & \bar{f}(\theta)(\ln r)^{4}+O(s^{8}).
\end{align*}

Also $\ln(1-r^{2})-\ln(1-r_{1}^{2})=\ln(1+\frac{r_{1}^{2}-r^{2}}{1-r_{1}^{2}})=\frac{r_{1}^{2}-r^{2}}{1-r_{1}^{2}}+O(s^{8})$.
Noticing that $h=o(s^{4}),$ by (\ref{eq:key identity}) and (\ref{eq:s^5}),
we get the conclusion. 
\end{proof}
We now proceed to obtain derivative estimates. 
\begin{lem}
\label{lem: r tangential derivative}Assume that $u(x)\in C^{2}(D\backslash\{p_{1,},\cdots,p_{k}\})$
satisfies Condition (2) in Definition 1 and we use notations given
as above. For $r=r(t,\theta),$ we have $r_{\theta}=O(s^{5})$ and
$r_{\theta\theta}=O(s^{5})$. 
\end{lem}

\begin{proof}
Taking derivatives to both sides of (\ref{eq:key identity}), we have

\begin{equation}
0=\frac{2r}{1-r^{2}}r_{\theta}+\nabla_{\theta}f\cdot s{}^{4}-\frac{f}{r}s^{3}r_{\theta}+h_{r}r_{\theta}+h_{\theta}.\label{eq:r first tangent derivative}
\end{equation}
Thus, noting the asymptotic behavior of $h$ and $h_{\theta}$ by
Definition \ref{def:AHS}, we get $r_{\theta}=O(s^{5})$. Taking further
derivative to (\ref{eq:r first tangent derivative}), we obtain 
\[
0=(\frac{2r}{1-r^{2}})_{r}r_{\theta}^{2}+\frac{2r}{1-r^{2}}r_{\theta\theta}+\nabla_{\theta\theta}f\cdot s^{4}-\frac{\nabla_{\theta}f}{r}s^{3}r_{\theta}--\frac{f}{r}s^{3}r_{\theta\theta}+h_{r}r_{\theta\theta}+h_{rr}r_{\theta}^{2}+h_{\theta\theta}.
\]
Thus
\[
r_{\theta\theta}=O(s)[O(s^{4})+h_{rr}r_{\theta}^{2}+h_{\theta\theta}]=O(s^{5})
\]
 by Condition (2) in Definition \ref{def:AHS}.
\end{proof}
\begin{lem}
\label{lem: first derivative} Assume that $u(x)\in C^{2}(D\backslash\{p_{1,},\cdots,p_{k}\})$
satisfies Condition (2) in Definition 1 and we use notations given
as above. By the asymptotic behavior of $u,$ we have 
\[
\nabla_{r}u=\nabla_{r}w+O(s^{3})=\frac{1}{s}[1+O(s)],
\]
\[
\nabla_{\theta}u=s^{4}\nabla_{\theta}f+o(s{}^{4})=O(s^{4}),
\]
\[
\nabla_{\theta_{i}\theta_{j}}u=O(s{}^{4}),
\]
\begin{equation}
|\nabla u|=|u_{r}|\big(1+O(s^{5})\big).\label{eq:add11}
\end{equation}
\end{lem}

The proof of Lemma \ref{lem: first derivative} is straightforward
so we omit it. Lemma \ref{lem: first derivative} implies that we
may approximate $|\nabla u|$ by $|\partial_{r}u|$.

To discuss geometry near $L(t)$, we define following maps: first
let $\pi$ be the projection map $\mathbb{R}^{4}\backslash\{O\}\to S^{3}:x\to\pi(x)=\frac{x}{|x|}$
and $i_{t}:L(t)\to\mathbb{R}^{4}\backslash\{O\}$ be the inclusion
map. Then for $d\theta$ being the volume form on $S^{3}$, $(\pi\circ i_{t})^{*}(d\theta)$
is then a volume form on $L(t)$ for $t$ large. To simplify the notation,
we simply write $(\pi\circ i_{t})^{*}(d\theta)$ as $d\theta$ when
no confusion arises. We are now ready to give the following estimate
of geometric terms. 
\begin{lem}
\label{lem: mean curavture and volume}Assume that $u(x)\in C^{2}(D\backslash\{p_{1,},\cdots,p_{k}\})$
satisfies Condition (2) in Definition 1 and we use notations given
as above. For the mean curvature $H$ of $L(t),$ we have, as $t\to\infty,$
\[
|H-\frac{3}{r_{1}}|=O(s^{5}).
\]
On the level set $L(t),$ we have the volume form $dl$ 
\[
dl=r_{1}^{3}(t)(\pi\circ i_{t})^{*}(d\theta)[1+O(s^{5})]=r^{3}(t,\theta)(\pi\circ i_{t})^{*}(d\theta)[1+O(s^{5})].
\]
\end{lem}

\begin{proof}
Using the polar coordinates, we denote $x\in\mathbb{R}^{4}$ also
as $x=r\theta$, $\theta\in S^{3}.$ Then for any point $x\in L(t)\subset B_{1}$,
the outer normal vector of $L(t)$, denoted as $\boldsymbol{n}$,
can be computed using Lemma \ref{lem: first derivative} 
\begin{equation}
n=\frac{\nabla u}{|\nabla u|}=\frac{x}{r}+O(s^{5})=\frac{x}{r_{1}}+O(s^{5}).\label{eq:add6}
\end{equation}

According to Lemma \ref{lem: first derivative}, we may find local
coordinate $\{\eta^{\alpha}\},$ $\alpha=1,2,3$ near $\theta\in S^{3}$.
Then $\eta^{\alpha}$ can be extended to an open set $W\subset B_{1}\backslash\{O\}$
as $\pi^{*}\eta^{\alpha}$, which we write as $\eta^{\alpha}$ for
simplicity. We have then $|\frac{\partial}{\partial\eta^{\alpha}}(y)|=|y|\cdot|\frac{\partial}{\partial\eta^{\alpha}}(\pi(y))|$
for $y\in W$. Consider the natural orthogonal projection map $p^{\bot}:T\mathbb{R}^{4}\to TL(t)$.
Now $\{p^{\bot}(\frac{\partial}{\partial\eta^{\alpha}})\}$ is a local
basis of $TL(t)$. Furthermore, by Lemma \ref{lem: first derivative},
\[
|\frac{\partial}{\partial\eta^{\alpha}}-p^{\bot}(\frac{\partial}{\partial\eta^{\alpha}})|=O(s^{5}).
\]
By Lemma \ref{lem: r tangential derivative} and \ref{lem: first derivative},
we estimate the first and second fundamental forms of $L(t)$ as following
\begin{align*}
h_{\alpha\beta} & =<p^{\bot}(\frac{\partial}{\partial\eta^{\alpha}})p^{\bot}(\frac{\partial}{\partial\eta^{\beta}})x,\boldsymbol{n}>\\
 & =<[\frac{\partial^{2}}{\partial\eta^{\alpha}\partial\eta^{\beta}}r(t,\theta)]\theta+\frac{\partial r}{\partial\eta^{\alpha}}\frac{\partial\theta}{\partial\eta^{\beta}}+\frac{\partial r}{\partial\eta^{\beta}}\frac{\partial\theta}{\partial\eta^{\alpha}}+r\frac{\partial^{2}\theta}{\partial\eta^{\alpha}\partial\eta^{\beta}},\boldsymbol{n}>+O(s^{5})\\
 & =r<\frac{\partial^{2}\theta}{\partial\eta^{\alpha}\partial\eta^{\beta}},\frac{x}{r}>+O(s^{5}),
\end{align*}
and similarly, 
\begin{align}
g_{\alpha\beta} & =<p^{\bot}(\frac{\partial}{\partial\eta^{\alpha}})x,p^{\bot}(\frac{\partial}{\partial\eta^{\beta}})x>=r^{2}<\frac{\partial\theta}{\partial\eta^{\alpha}},\frac{\partial\theta}{\partial\eta^{\beta}}>+O(s^{5}).\label{eq:first fundamental}\\
\nonumber 
\end{align}
Therefore we estimate the mean curvature of $L(t)$, 
\[
H=g^{\alpha\beta}h_{\alpha\beta}=\frac{1}{r}H_{S^{3}}+O(s^{5})=\frac{3}{r_{1}(t)}+O(s^{5}).
\]
Furthermore, note that $dx=r^{3}dr\pi^{*}(d\theta)$, and $dl=i_{t}^{*}(\iota(n)dx).$
Using (\ref{eq:add6}), we have $|n-\frac{x}{r}|=O(s^{5}).$ We may
then estimate the volume form of $L(t)$ as follows 
\begin{equation}
dl=i_{t}^{*}(\iota(n)dx)=i_{t}^{*}r^{3}\pi^{*}(d\theta)[1+O(s^{5})]=r_{1}^{3}(t)(\pi\circ i_{t})^{*}d\theta[1+O(s^{5})].\label{eq:b6}
\end{equation}
As a consequence, we also have 
\begin{equation}
|L(t)|=r_{1}^{3}(t)|S^{3}|[1+O(s^{5})].\label{eq:volume approx}
\end{equation}
\end{proof}
With all local point-wise estimates in place, we are ready to compute
integrals that have appeared in our quasi-local mass. First, we have 
\begin{lem}
\label{lem:rw'^3}Assume that $u(x)\in C^{2}(D\backslash\{p_{1,},\cdots,p_{k}\})$
satisfies Condition (2) in Definition 1 and we use notations given
as above. For any fixed $t$ very large, let $x=(r(t,\theta),\theta)\in L(t),$
and $r_{1}=r_{1}(t)$, then 
\[
\fint_{L(t)}r^{3}(t,\theta)(w'(r))^{3}(\pi\circ i_{t})^{*}(d\theta)=\big(r_{1}w^{\prime}(r_{1})\big)^{3}+o(s).
\]
\end{lem}

\begin{proof}
By definition of $w$, 
\begin{align*}
 & w'\mid_{r=r(t,\theta)}\\
= & \frac{1}{1-r}-\frac{1}{1+r}\\
= & \frac{1}{1-r_{1}-\varepsilon}-\frac{1}{1+r_{1}+\varepsilon}\\
= & \frac{\varepsilon}{\left(1-r_{1}\right)\left(1-r_{1}-\varepsilon\right)}+\frac{\varepsilon}{(1+r_{1})(1+r_{1}+\varepsilon)}+w^{\prime}(r_{1}).
\end{align*}

We get 
\begin{align}
 & rw'\mid_{r=r(t,\theta)}\label{eq:first derivative difference}\\
= & \left(r_{1}+\varepsilon\right)\bigg\{\frac{\varepsilon}{\left(1-r_{1}\right)\left(1-r_{1}-\varepsilon\right)}+\frac{\varepsilon}{(1+r_{1})(1+r_{1}+\varepsilon)}+w^{\prime}(r_{1})\bigg\}\\
= & r_{1}w^{\prime}\left(r_{1}\right)+r_{1}\left[\frac{\varepsilon}{\left(1-r_{1}\right)\left(1-r_{1}-\varepsilon\right)}+\frac{\varepsilon}{(1+r_{1})\left(1+r_{1}+\varepsilon\right)}\right]\nonumber \\
 & +\varepsilon\left[\frac{\varepsilon}{\left(1-r_{1}\right)\left(1-r_{1}-\varepsilon\right)}+\frac{\varepsilon}{(1+r_{1})\left(1+r_{1}+\varepsilon\right)}+w^{\prime}\left(r_{1}\right)\right]\nonumber \\
=: & r_{1}w^{\prime}\left(r_{1}\right)+F_{1}+F_{2},\nonumber 
\end{align}
where 
\[
F_{1}=r_{1}\left[\frac{\varepsilon}{\left(1-r_{1}\right)\left(1-r_{1}-\varepsilon\right)}+\frac{\varepsilon}{(1+r_{1})\left(1+r_{1}+\varepsilon\right)}\right]
\]
 and 
\[
F_{2}=\varepsilon\left[\frac{\varepsilon}{\left(1-r_{1}\right)\left(1-r_{1}-\varepsilon\right)}+\frac{\varepsilon}{(1+r_{1})\left(1+r_{1}+\varepsilon\right)}+w^{\prime}\left(r_{1}\right)\right].
\]
It is clear that $r_{1}w^{\prime}\left(r_{1}\right)=O(\frac{1}{s})$,
$F_{1}=O(s^{3})$ and $F_{2}=O(s^{4})$ by Lemma \ref{lem:basic estimate}
and Corollary \ref{cor:r-r1}. Thus, 
\begin{align}
\fint_{L(t)}r^{3}(w')^{3}(\pi\circ i_{t})^{*}d\theta & =\fint_{L(t)}[r_{1}w^{\prime}(r_{1})+F_{1}+F_{2}]^{3}(\pi\circ i_{t})^{*}d\theta\label{eq:the most difficult tem}\\
 & =\fint_{L(t)}\{\big(r_{1}w^{\prime}(r_{1})\big)^{3}+3F_{1}[r_{1}w^{\prime}(r_{1})]^{2}\}(\pi\circ i_{t})^{*}d\theta+O(s^{2})\nonumber \\
 & =\big(r_{1}w^{\prime}(r_{1})\big)^{3}+3[r_{1}w^{\prime}(r_{1})]^{2}\fint_{L(t)}F_{1}(\pi\circ i_{t})^{*}d\theta+O(s^{2}).\nonumber 
\end{align}

We compute the second term. Noting that by Lemma \ref{lem:basic estimate}
and Corollary \ref{cor:r-r1}, 

\begin{align}
\fint_{L(t)}F_{1}(\pi\circ i_{t})^{*}d\theta & =\fint_{L(t)}r_{1}\left[\frac{\varepsilon}{\left(1-r_{1}\right)\left(1-r_{1}-\varepsilon\right)}+\frac{\varepsilon}{(1+r_{1})\left(1+r_{1}+\varepsilon\right)}\right](\pi\circ i_{t})^{*}d\theta\label{eq:add21}\\
 & =\frac{r_{1}}{1-r_{1}}\fint_{L(t)}\frac{\varepsilon\ (\pi\circ i_{t})^{*}d\theta}{1-r_{1}-\varepsilon}+\frac{r_{1}}{1+r_{1}}\fint_{L(t)}\frac{\varepsilon\ (\pi\circ i_{t})^{*}d\theta}{(1+r_{1}+\varepsilon)}\nonumber \\
 & =\frac{r_{1}}{1-r_{1}}\fint_{L(t)}(\frac{\varepsilon}{1-r_{1}}+\frac{\varepsilon^{2}}{(1-r_{1}-\varepsilon)(1-r_{1})}\big)(\pi\circ i_{t})^{*}d\theta+O(s^{5}).\nonumber 
\end{align}

Using Corollary \ref{cor:r-r1} and (\ref{eq:b6}), and noting that
$\int_{S^{3}}\bar{f}d\theta=0,$ we get 
\begin{equation}
\fint_{L(t)}\varepsilon\ (\pi\circ i_{t})^{*}d\theta=\fint_{L(t)}[-\frac{(1-r_{1}^{2})}{2r_{1}}\bar{f}s'^{4}+o(s^{5})](\pi\circ i_{t})^{*}(d\theta)=o(s^{5}).\label{eq:add22}
\end{equation}

Therefore, by (\ref{eq:add21}) and (\ref{eq:add22}), we have 
\begin{equation}
\fint_{L(t)}F_{1}=o(s^{3}).\label{eq:add23}
\end{equation}
We then apply Lemma \ref{lem:basic estimate} and (\ref{eq:the most difficult tem})
to (\ref{eq:add23}) and obtain 
\begin{equation}
\fint_{L(t)}r^{3}(w')^{3}\ (\pi\circ i_{t})^{*}d\theta=\big(r_{1}w^{\prime}(r_{1})\big)^{3}+o(s).\label{eq:z3's 1}
\end{equation}
\end{proof}
Now we estimate $z(t)$. 
\begin{lem}
\label{lem:z}Assume that $u(x)\in C^{2}(D\backslash\{p_{1,},\cdots,p_{k}\})$
satisfies Condition (2) in Definition 1 and we use notations given
as above. We have 
\[
z(t)=r_{1}w^{\prime}(r_{1})+O(s^{3}),
\]
\[
z^{3}(t)=\big(r_{1}w^{\prime}(r_{1})\big)^{3}+12(r_{1}w^{\prime}(r_{1}))^{2}(\ln r_{1})^{3}\fint_{S^{3}}f(\theta)d\theta+o(s),
\]

and 
\[
z^{4}(t)=\big(r_{1}w^{\prime}(r_{1})\big)^{4}+16(\ln r_{1})^{3}\big(r_{1}w^{\prime}(r_{1})\big)^{3}\fint_{S^{3}}f(\theta)d\theta+o(1).
\]
\end{lem}

\begin{proof}
We use Lemma \ref{lem: mean curavture and volume} to see that 
\[
\fint_{L(t)}|\nabla u|^{3}dl_{t}=\fint_{L(t)}|\nabla u|^{3}r^{3}(t,\theta)(\pi\circ i_{t})^{*}d\theta(1+O(s^{5})).
\]
Then, by Lemma \ref{lem: first derivative}, we get 
\[
\fint_{L(t)}|\nabla u|^{3}r^{3}(t,\theta)(\pi\circ i_{t})^{*}d\theta\cdot O(s^{5})\le O(s^{2}).
\]
We may then use (\ref{eq:new1}) to compute $z.$ Let $e(x)=-f(\theta)(\ln r)^{4}-h(x)$.
Then $u=w-e,$ and

\begin{align*}
|\nabla u|^{3} & =\big((w')^{2}+|\nabla e|^{2}-2w{}_{r}\nabla_{r}e\big)^{\frac{3}{2}}\\
 & =(w')^{3}\bigg\{1+\frac{3}{2}(\frac{|\nabla e|^{2}-2w'\nabla_{r}e}{(w')^{2}})+O\bigg(\big(\frac{3}{2}(\frac{|\nabla e|^{2}-2w'\nabla_{r}e}{(w')^{2}})\big)^{2}\bigg)\bigg\}\\
 & =(w')^{3}+\frac{3}{2}w'(|\nabla e|^{2}-2w'\nabla_{r}e)+O(\frac{(|\nabla e|^{2}-2w'\nabla_{r}e)^{2}}{w'}),\\
\end{align*}
which leads to 
\begin{align}
 & \fint_{L(t)}r^{3}|\nabla u|^{3}(\pi\circ i_{t})^{*}d\theta\label{eq:z3 asmpototic}\\
 & =\fint_{L(t)}r(t,\theta)^{3}(w')^{3}(\pi\circ i_{t})^{*}d\theta+\frac{3}{2}r^{3}w'(|\nabla e|^{2}-2w'\nabla_{r}e)(\pi\circ i_{t})^{*}d\theta+O(r^{3}\frac{(|\nabla e|^{2}-2w'\nabla_{r}e)^{2}}{w'}).
\end{align}
By Condition (2) in Definition \ref{def:AHS}, $O(r^{3}\frac{(|\nabla e|^{2}-2w'\nabla_{r}e)^{2}}{w'})=O(s^{5})$.

By (\ref{eq:first derivative difference}), (\ref{eq:volume approx})
and Condition (2) in Definition \ref{def:AHS}, 
\begin{align}
 & \fint_{L(t)}\frac{3}{2}r^{3}w'(|\nabla e|^{2}-2w'\nabla_{r}e)(\pi\circ i_{t})^{*}d\theta\label{eq:z3's 2}\\
 & =o(s)+\fint_{L(t)}3r^{3}(w')^{2}f(\theta)\nabla_{r}((\ln r)^{4})(\pi\circ i_{t})^{*}d\theta\\
 & =o(s)+\fint_{L(t)}12\bigg(r_{1}^{2}(w^{\prime}(r_{1}))^{2}+O(s^{2})\bigg)f(\theta)(\ln r)^{3}(\pi\circ i_{t})^{*}d\theta\nonumber \\
 & =o(s)+\fint_{L(t)}12\bigg(r_{1}^{2}(w^{\prime}(r_{1}))^{2}+O(s^{2})\bigg)f(\theta)\big((\ln r_{1})^{3}+O(s^{7})\big)(\pi\circ i_{t})^{*}d\theta\nonumber \\
 & =o(s)+12\big(r_{1}w^{\prime}(r_{1}))^{2}(\ln r_{1})^{3}\fint_{S^{3}}f(\theta)d\theta\nonumber \\
 & =12r_{1}^{2}(w^{\prime}(r_{1}))^{2}(\ln r_{1})^{3}\fint_{S^{3}}f(\theta)d\theta+o(s).\nonumber 
\end{align}
By (\ref{eq:z3's 1})(\ref{eq:z3 asmpototic})(\ref{eq:z3's 2}) and
Lemma \ref{lem:rw'^3}, we have

\begin{align*}
z^{3} & =\fint_{L(t)}r^{3}|\nabla u|^{3}(\pi\circ i_{t})^{*}d\theta+O(s^{2})\\
 & =\big(r_{1}w^{\prime}(r_{1})\big)^{3}+12(r_{1}w'(r_{1}))^{2}(\ln r_{1})^{3}\fint_{S^{3}}f(\theta)d\theta+o(s).\\
\end{align*}
Furthermore, 
\begin{align*}
z^{4} & =(\fint_{S^{3}}r^{3}|\nabla u|^{3})^{4/3}\\
 & =\big(r_{1}w^{\prime}(r_{1})\big)^{4}\bigg(1+16\frac{(\ln r_{1})^{3}}{r_{1}w'(r_{1})}\fint_{S^{3}}f(\theta)d\theta+o(s^{4})\bigg)\\
 & =\big(r_{1}w^{\prime}(r_{1})\big)^{4}+16(\ln r_{1})^{3}\big(r_{1}w^{\prime}(r_{1})\big)^{3}\fint_{S^{3}}f(\theta)d\theta+o(1).
\end{align*}
And 
\begin{align*}
z & =r_{1}w^{\prime}(r_{1})\bigg(1+4\frac{(\ln r_{1})^{3}}{r_{1}w^{\prime}(r_{1})}\fint_{S^{3}}f(\theta)d\theta+o(s^{4})\bigg)\\
 & =r_{1}w^{\prime}(r_{1})+O(s^{3}).
\end{align*}
\end{proof}
Now let us compute $\int_{L(t)}H|\nabla u|^{2}dl.$ 
\begin{lem}
\label{lem: mean curvature integ}Assume that $u(x)\in C^{2}(D\backslash\{p_{1,},\cdots,p_{k}\})$
satisfies Condition (2) in Definition 1 and we use notations given
as above. For $f\in C^{2}(S^{3})$, 
\[
\fint_{L(t)}H|\nabla u|^{2}dl=3\big(r_{1}w^{\prime}(r_{1})\big)^{2}+O(s^{2}).
\]
\end{lem}

\begin{proof}
By Lemma \ref{lem: first derivative} and (\ref{eq:first derivative difference}),

\begin{align*}
H|\nabla u|^{2}r^{3} & =r^{3}(\frac{3}{r_{1}}+O(s^{5}))(u_{r}^{2}+\frac{1}{r^{2}}u_{\theta}^{2})\\
 & =3r^{2}u_{r}^{2}+O(s^{3})\\
 & =3r^{2}(w'+O(s^{3}))^{2}+O(s^{3})\\
 & =3r^{2}(w')^{2}+O(s^{2})\\
 & =3\big(r_{1}w^{\prime}\left(r_{1}\right)+F_{1}+F_{2}\big)^{2}+O(s^{2})\\
 & =3\big(r_{1}w^{\prime}(r_{1})\big)^{2}+O(s^{2}).
\end{align*}

Combining with Lemma \ref{lem: mean curavture and volume}, we get

\begin{align*}
\fint_{L(t)}H|\nabla u|^{2}dl & =\fint_{L(t)}H|\nabla u|^{2}r^{3}(\pi\circ i_{t})^{*}d\theta(1+O(s^{5}))\\
 & =\fint_{L(t)}[3\big(r_{1}w^{\prime}(r_{1})\big)^{2}+O(s^{2})](\pi\circ i_{t})^{*}d\theta(1+O(s^{5}))\\
 & =3\big(r_{1}w^{\prime}(r_{1})\big)^{2}+O(s^{2}).
\end{align*}
\end{proof}
From Lemma \ref{lem:z} and Lemma \ref{lem: mean curvature integ},
we obtain 
\begin{lem}
\label{lem:mixed terms} Assume that $u(x)\in C^{2}(D\backslash\{p_{1,},\cdots,p_{k}\})$
satisfies Condition (2) in Definition 1 and we use notations given
as above. We have
\begin{align}
\frac{2}{9}z\fint_{L(t)}H|\nabla u|^{2}dl & =\frac{2}{3}\big(r_{1}w^{\prime}(r_{1})\big)^{3}+O(s).\label{eq:mixed term}
\end{align}
\end{lem}

\begin{proof}
By Lemmas \ref{lem:z} and \ref{lem: mean curvature integ}, we have
\[
\frac{2}{9}z\fint_{L(t)}H|\nabla u|^{2}=\frac{2}{9}\big(r_{1}w^{\prime}(r_{1})+O(s^{3})\big)\big(3\big(r_{1}w^{\prime}(r_{1})\big)^{2}+O(s^{2})\big).
\]
\end{proof}
Next, we estimate $C(t)$. 
\begin{lem}
\label{lem:C(t)} Assume that $u(x)\in C^{2}(D\backslash\{p_{1,},\cdots,p_{k}\})$
satisfies Condition (2) in Definition 1 and we use notations given
as above. We have
\[
C(t)=\frac{r_{1}^{4}e^{4w(r_{1})}}{4}+\fint_{S^{3}}f(\theta)d\theta\ (\frac{2}{1-r_{1}^{2}})^{4}(\ln r_{1})^{4}r_{1}^{4}(t)+O(s).
\]
\end{lem}

\begin{proof}
For $t=u_{1}(r_{1}(t))=u(r(t,\theta),\theta)$, noting $u_{1}(r_{1})=w(r_{1})+\frac{(\ln r_{1})^{4}}{|S^{3}|}\int_{S^{3}}f(\theta)d\theta$,
Lemma \ref{cor:r-r1} and Lemma \ref{lem:basic estimate}, we use
the polar coordinate to compute the integrals over regions of $\mathbb{R}^{4}$,

\begin{align}
C(t) & =\frac{1}{|S^{3}|}e^{4t}|\{u<t\}|\label{eq:volume estimate}\\
 & =\frac{1}{|S^{3}|}e^{4u_{1}(r_{1}(t))}\int_{S^{3}}\int_{\gamma\leq r(t,\theta)}\gamma^{3}d\gamma\pi^{*}(d\theta)\\
 & =\frac{1}{|S^{3}|}e^{4\fint_{S^{3}}f(\theta)d\theta(\ln r_{1})^{4}}e^{4w(r_{1})}\int_{S^{3}}\frac{1}{4}r^{4}(t,\theta)d\theta\nonumber \\
 & =\frac{1}{|S^{3}|}e^{4w(r_{1})}\bigg(1+4\fint_{S^{3}}f(\theta)d\theta\ (\ln r_{1})^{4}+O(s^{8})\bigg)\bigg(\int_{S^{3}}\frac{1}{4}r_{1}^{4}(t)d\theta+O(\varepsilon(t,\theta))\bigg)\nonumber \\
 & =\frac{r_{1}^{4}e^{4w(r_{1})}}{4}+\fint_{S^{3}}f(\theta)d\theta\ (\frac{2}{1-r_{1}^{2}})^{4}(\ln r_{1})^{4}r_{1}^{4}+O(s).\nonumber 
\end{align}
\end{proof}
Finally, we are ready to prove Theorem \ref{thm:negative side}. 
\begin{proof}
First, recall our quasi-local mass 
\[
m(t)=\frac{1}{5}[\frac{1}{4}z^{4}+\frac{2}{9}z\fint_{L(t)}H|\nabla u|^{2}+\frac{1}{3}z^{3}+\frac{1}{3}\fint_{L(t)}H|\nabla u|^{2}-C(t)].
\]
Second, recall $w=\log\frac{2}{1-r^{2}}$, which satisfies the following
differential equation: 
\begin{equation}
\frac{1}{4}(rw')^{4}+(rw')^{3}+(rw')^{2}-\frac{r^{4}}{4}e^{4w(r)}=0.\label{eq:add8}
\end{equation}
When $t\to\infty,$ we have $s\to0^{+},$ and $r_{1}\to1.$ We use
Lemma \ref{lem:z}, Lemma \ref{lem: mean curvature integ}, Lemma
\ref{lem:mixed terms}, Lemma \ref{lem:C(t)} and (\ref{eq:add8})
to compute

\[
\lim_{t\rightarrow\infty}m(t)=\frac{1}{5}[-4\fint_{S^{3}}f(\theta)d\theta-\fint_{S^{3}}f(\theta)d\theta]=-\fint_{S^{3}}f(\theta)d\theta=-m(M,g).
\]
We have thus finished the proof. 
\end{proof}

\end{document}